\documentclass[12pt]{amsart}
\usepackage{latexsym,color,amsmath,amsthm,amssymb,amscd,amsfonts}

\usepackage{hyperref}
\usepackage{graphicx}

\newcommand{\R}{\mathbb R}
\newcommand{\K}{\mathcal K}

\newcommand{\conv}{\mathrm{conv}}

\newcommand{\inter}{\mathrm{int}}
\newcommand{\lin}{\mathrm{span}}

\newtheorem{thm}{Theorem}[section]
\newtheorem{lemma}[thm]{Lemma}

\newtheorem{conj}[thm]{Conjecture}

\newtheorem{rmk}[thm]{Remark}

\begin{document}


\title[Estimates average of functions by a center]{Estimating the average of functions with convexity properties by means of a new center}

\author[B. Gonz\'alez Merino]{Bernardo Gonz\'alez Merino}
\address{Departamento de Did\'actica de las Ciencias Matem\'aticas y Sociales, Facultad de Educaci\'on, Universidad de Murcia, 30100-Murcia, Spain}\email{bgmerino@um.es}

\thanks{
This research is a result of the activity developed within the framework of the Programme in Support of Excellence Groups of the Regi\'on de Murcia, Spain, by Fundaci\'on S\'eneca, Science and Technology Agency of the Regi\'on de Murcia. Partially supported by Fundaci\'on S\'eneca project 19901/GERM/15, Spain, and by MICINN Project PGC2018-094215-B-I00 Spain.}

\date{\today}\maketitle

\begin{abstract}
In this article we show the following result: if $C$ is an $n$-dimensional convex and compact subset, $f:C\rightarrow[0,\infty)$ is concave, 
and $\phi:[0,\infty)\rightarrow[0,\infty)$ is a convex function with $\phi(0)=0$, we then characterize the class of sets and concave functions that attain the supremum 
\[
\sup_{C,f}\int_C\phi(f(x))dx,
\]
where the supremum ranges over all sets $C$ with $n$-dimensional volume $|C|=c$ and the additional condition that $f(x_{C,f})=k$
for some point $x_{C,f}\in C$ that we introduce in the article, for two non-negative constants $c,k>0$. 

As a consequence, we extend some results of Milman and Pajor in \cite{MP} and some in \cite[Thm. 1.2]{GoMe}. Besides, we also obtain some new estimates on the volume of particular sections of a convex set $K$ passing through a new point of $K$.
\end{abstract}
\date{\today}\maketitle

\section{Introduction}

Let $C$ be a \emph{convex body}, i.e. a convex and compact set of $\R^n$. Let $\K^n$ be the set of all convex bodies. Let $|C|$ be the \emph{$n$-dimensional volume} (or Lebesgue measure) of $C$. 

The classical Hermite-Hadamard inequality (proven independenty by Hermite 1881 and Hadamard 1893 in the 1-dimensional case), is a direct consequence of the Jensen's inequality (see \cite{J}).  See \cite{DP} (and \cite{CalCar} or \cite{St})
and the references on it for other historical considerations and a comprehensive and complete view of this type of inequalities.
It states that for any $C\in\mathcal K^n$ and $f:C\rightarrow[0,\infty)$ concave, then
\begin{equation}\label{eq:HH_Jensen}
\frac{1}{|C|}\int_Cf(x)dx\leq f\left(x_C\right),
\end{equation}
where $x_C=\frac{1}{|C|}\int_Cxdx$ is the \emph{centroid} of $C$.
Very recently in \cite{GoMe}, it has been generalized in $n$-dimensions for $0$-symmetric sets. 
More precisely, if $C\in\mathcal K^n$ with $C=-C$, $f:C\rightarrow[0,\infty)$
is concave, and $\phi:[0,\infty)\rightarrow[0,\infty)$ is a convex function with $\phi(0)=0$, then
\begin{equation}\label{eq:Gener_HH_0-symm}
\frac{1}{|C|}\int_C\phi(f(x))dx\leq\frac{1}{2}\int_{-1}^1\phi(f(0)(1+t))dt.
\end{equation}

The mean value of a function measured in $C$ (the left-term in \eqref{eq:HH_Jensen}) has repeatedly appeared
during the development of different topics of Analysis and Geometry (cf.~\cite{HLP}).
Berwald \cite[Sect.~7]{Ber} studied monotonicity relations of $L_p$ means of concave functions over convex compact domains. He showed that for any $C\in\K^n$ and $f:C\rightarrow[0,\infty)$ concave, then
\[
t\rightarrow\left(\frac{{n+t\choose n}}{|C|}\int_Cf(x)^t\right)^\frac{1}{t}
\]
is decreasing for $t>0$ (see \cite{GZ} for an extension to $t>-1$, see also \cite{ABG} and \cite[Sect.~7]{AAGJV} for a translation of it). It represents a reverse result to the Jensen's inequality \cite{J}, i.e. 
\[
t\rightarrow\left(\frac{1}{|C|}\int_Cf(x)^t\right)^\frac{1}{t},
\]
which is increasing for $t>0$. Borell \cite{Bor} did a step further by showing some convexity relations in the same regard (Thms.~1 and 2).
See also Milman and Pajor \cite[2.6]{MP2} and \cite[Sect.~5]{GNT},
or the Hardy-Littlewood Maximal Function (cf.~\cite{Me}).

A natural question that arises from \eqref{eq:Gener_HH_0-symm} is to seek for possible extensions to the case of not necessarily symmetric $C$. We can find partial answers in the literature. For instance, Milman and Pajor \cite[Lem. 1.1]{MP} proved that if $f:\mathbb R^n\rightarrow[0,\infty)$ is an integrable \emph{log-concave function} (i.e. such that $\log f$ is concave) and $\mu:\mathbb R^n\rightarrow[0,\infty)$ is a probability measure,
then 
\begin{equation}\label{eq:MP-Ineq}
\int_{\mathbb R^n}f(x)d\mu(x)\leq f\left(\int_{\mathbb R^n}x\frac{f(x)}{\int_{\mathbb R^n}f(z)d\mu(z)}d\mu(x)\right).
\end{equation}
This inequality is independent of \eqref{eq:Gener_HH_0-symm}. However, certain situations can be covered by the first (resp. second) one and not by the second (resp. first),  for instance, when $C$ is not symmetric and $\phi(f)$ is log-concave (resp. when
$C$ is 0-symmetric but $\phi(f)$ is not log-concave). As an example of the first, let $C$ be non-symmetric and $\phi(t)=t^\alpha$, for some $\alpha\geq 1$, since then $\phi(f)=f^\alpha$ is always log-concave. As an example of the second, let $C$ be 0-symmetric and $\phi(t)=e^{t^2}$, since then $\phi(f)=e^{f^2}$ is in general not log-concave. 

For a set $X\subset\R^n$, its \emph{convex hull}, $\conv(X)$, is the smallest convex set containing $X$. Let $B_n$ be the \emph{Euclidean unit ball},
and let $\kappa_n=|B_n|$. Let $\mathrm{Gr}(i,\mathbb R^n)$ be the set of all \emph{$i$-dimensional linear subspaces} contained in $\R^n$. Let $e_1,\dots,e_n$ be the \emph{canonical basis} of $\R^n$.
A \emph{generalized truncated cone} is a convex body of the form $C=\conv(C_0\cup(x_0+\rho C_0))$, for some $C_0\in\K^n$, $C_0\subset H$ for some $H\in\mathrm{Gr}(n-1,\mathbb R^n)$, some $x_0\in \R^n$ and $\rho\geq 0$. We say that $C_0$ and $x_0+\rho C_0$ are the \emph{bases} of $C$. A \emph{generalized cone} is a generalized truncated cone with $\rho=0$ (i.e. one of its bases is a vertex).

Our main contribution in this work is to investigate a generalization of \eqref{eq:Gener_HH_0-symm} which covers those missing cases described above.
Moreover, it represents an inequality of the type \eqref{eq:MP-Ineq} for bounded domains. To do so, we make use of a \emph{new point}, $x_{C,f}$, which is associated to each $C\in\mathcal K^n$ and $f:C\rightarrow[0,\infty)$ concave. Since its definition is rather technical, we refer, for its definition, to Section \ref{sec:Main_result_Proofs} and in particular to Lemma \ref{lem:ExistenceOfx_C}.

\begin{thm}\label{thm:General_Result}
Let $C\in\mathcal K^n$, $f:C\rightarrow[0,\infty)$ be concave, $\phi:[0,\infty)\rightarrow[0,\infty)$ be convex, with $\phi(0)=0$. Then
\[
\int_C\phi(f(x))dx \leq \max_{R}\int_R\phi(g(x))dx,
\]
where the maximum ranges over all generalized truncated cones $R$ with $|R|=|C|$, and $g$ is an affine function with $g(x_{R,g})=f(x_{C,f})$ and which is zero in one of the bases of $R$. 

Moreover, if $\phi$ is strictly convex, then equality holds if and only if $C$ is a certain generalized truncated cone, and $f$ is an affine function with value zero at some of the bases of $C$.
\end{thm}

Notice that the maximum above can be computed via Fubini's formula (once assuming after a rotation of $R$ and a change of variables that $R=\conv(\{(r_m (B_n\cap H))\cup(e_1+(r_m+m)(B_n\cap H))\})$, where $H=\lin(\{e_2,\dots,e_n\})$) so that it equals
\begin{equation}\label{eq:reduced_formula}
\max_{m\in[-m_0,m_0]}\int_0^1\phi\left(\frac{f(x_{C,f})t}{t_m}\right)\kappa_{n-1}(r_m+mt)^{n-1}dt,
\end{equation}
where $m_0=(cn/\kappa_{n-1})^{1/(n-1)}$, $r_m$ is a solution to the equation $mnc=\kappa_{n-1}((r_m+m)^{n}-r_m^n)$, $t_m=(((r_m^n+(r_m+m)^n)/2)^{1/n}-r_m)/m$ and $c=|C|$ (see Remark \ref{rmk:concrete_computation} and Section \ref{sec:Higher_dim_case} for details).

In some particular cases, we can find the truncated cone and the affine function that attains the maximum in Theorem \ref{thm:General_Result}.

\begin{thm}\label{thm:2_dim_case}
Let $C\in\mathcal K^2$, $f:C\rightarrow[0,\infty)$ be concave, and let $\alpha\geq 1$. Then
\[
\frac{1}{|C|}\int_Cf(x)^{\alpha}dx \leq \frac{2}{(\alpha+1)(\alpha+2)}\left(\frac{\sqrt{2}}{\sqrt{2}-1}\right)^\alpha f(x_{C,f})^\alpha.
\]
Moreover, equality holds if and only if $C$ is a triangle and $f$ is an affine function which is equal to $0$ in one of the edges of $C$.
\end{thm}

Observe that Theorem \ref{thm:2_dim_case} applied to $\phi(t)=t$ already gives back a sharp inequality that \emph{does not} recover the classical Hermite-Hadamard inequality \eqref{eq:HH_Jensen} (in contrast to \cite[Thm. 1.2]{GoMe} that recovers it). 

For any $X\subset\R^n$, its \emph{affine hull} (resp. \emph{linear hull}), $\mathrm{aff}(X)$ (resp. $\lin(X)$), is the smallest affine (resp. linear) subspace containing $X$. The \emph{dimension} of $X$, $\mathrm{dim}(X)$, is the dimension of $\mathrm{aff}(X)$. For any $C\in\K^n$ and $H\in\mathrm{Gr}(i,\mathbb R^n)$, the denote by $P_HC$ the \emph{orthogonal projection of $C$ onto $H$}. Let $X^\bot$ be the \emph{orthogonal subspace} to $X$. Moreover, if $C\in\K^n$ with $\mathrm{dim}(C)<n$, then $|C|$ denotes the volume of $C$ computed inside the ambient space $\mathrm{aff}(C)$. 

Using the same techniques as in \cite[Thm. 1.1]{GoMe}, Theorem \ref{thm:2_dim_case} gives us an upper bound
of the volume of $K$, for general $K\in\mathcal K^n$, in terms of the volumes of sections and projections of $K$ of orthogonal subspaces.

\begin{thm}\label{thm:volume_sections}
Let $K\in\mathcal K^n$ and $H\in\mathrm{Gr}(2,\mathbb R^n)$. Then 
\[
|K|\leq \frac{2}{n(n-1)}\left(\frac{\sqrt{2}}{\sqrt{2}-1}\right)^{n-2}|P_HK||K\cap(x_{x_{P_HK,f}}+H^\bot)|,
\]
where $f(x)=|K\cap(x+H^\bot)|^\frac{1}{n-2}$.

Moreover, equality holds if and only if 
there exist
a segment $S\subset H$, a point $x_0\in H\setminus\mathrm{aff}(S)$, and $Q\in\mathcal K^n$, $Q\subset H^\bot$, such that
\[
K=\{(1-\lambda)a+\lambda x_0+g((1-\lambda)a+\lambda x_0)+\lambda Q:a\in S,\,\lambda\in[0,1]\},
\]
for some $g:H\rightarrow\mathbb R^n$ linear function. 
\end{thm}
Spingarn \cite{S} and later Milman and Pajor \cite{MP} proved that replacing above $x_{P_HK,f}$ by $x_K$ allows to remove the constant in the right-hand side term (cf. also \cite[Thm. 1.1]{GoMe} or \cite{Fr} for other results of this type). 

If we apply Theorem \ref{thm:General_Result} in the planar case to $\phi(t)=e^t-1$, we obtain an upper bound for log-concave functions (see also \eqref{eq:MP-Ineq}). We compute the exact constant for particular values of $f(x_{C,f})=f_0$, since the study of 
the underlying monotonicity seems to be intractable and needs of numerical approximations.
\begin{thm}\label{thm:Exp_2_dim_case}
Let $C\in\mathcal K^2$ and $f:C\rightarrow[0,\infty)$ be concave. Then
\[
\frac{1}{|C|}\int_Ce^{\frac{f(x)}{f(x_{C,f})}}dx \leq \sqrt{2}(\sqrt{2}-1)\left(\frac{e^{\frac{\sqrt{2}}{\sqrt{2}-1}}}{\sqrt{2}}-1\right).
\]
Moreover, equality holds if and only if $C$ is a triangle and $f$ is an affine function which becomes zero in one of the edges of $C$.
\end{thm}

The technique developed by Milman and Pajor in \cite{MP} cannot be applied if $\phi(f(x))$ is not a log-concave function.
For instance, next theorem is obtained as a result of applying Theorem \ref{thm:General_Result} in the planar case with $\phi(t)=e^{t^2}-1$. 
\begin{thm}\label{thm:PhiConvNonLogConcave}
Let $C\in\mathcal K^2$, $f:C\rightarrow[0,\infty)$ be concave. Then
\[
\begin{split}
\frac{1}{|C|} & \int_Ce^{\frac{f(x)^2}{f(x_{C,f})^2}}dx \leq 
\frac{\sqrt{\pi}}{2}\mathrm{erfi}\left(\frac{\sqrt{2}}{\sqrt{2}-1}\right) \\
&+ \frac{1}{1-2\sqrt{2}}\left(\frac{e^{\left(\frac{\sqrt{2}}{\sqrt{2}-1}\right)^2}-1}{\left(\frac{\sqrt{2}}{\sqrt{2}-1}\right)^2}-\frac{\sqrt{\pi}}{2}\mathrm{erfi}\left(\frac{\sqrt{2}}{\sqrt{2}-1}\right)\right),
\end{split}
\]
where $\mathrm{erfi}(x)=-i\mathrm{erf}(ix)=\frac{2}{\sqrt{\pi}}\int_0^xe^{t^2}dt$ is the imaginary error function (in Wolfram Language) and $\mathrm{erf}(x)=(2/\sqrt{\pi})\int_0^{x}e^{-t^2}dt$ is the usual error function.

Moreover, equality holds if and only if $C$ is a triangle and $f$ is an affine function which has value $0$ in one of the edges of $C$.
\end{thm}

We briefly show other instances of Theorem \ref{thm:General_Result} in higher dimensions. However, and due to the technical difficulties to solve the underlying optimization problem, we have just considered here the corresponding Hermite-Hadamard inequality in the 3-dimensional case (see Section \ref{sec:Higher_dim_case} for further discussions on this).
\begin{thm}\label{thm:HH-3-dim-case}
Let $C\in\K^3$ and $f:C\rightarrow[0,\infty)$ be concave. Then
\[
\frac{1}{|C|}\int_Cf(x)dx \leq \frac{3\cdot 2^{1/3}}{2^{1/3}-1}f(x_{C,f}).
\]
Moreover, equality holds if and only if $C$ is a generalized cone and $f$ is an affine function which becomes zero at the base of $C$.
\end{thm}

After deriving the solutions of Theorems \ref{thm:2_dim_case}, \ref{thm:Exp_2_dim_case}, \ref{thm:PhiConvNonLogConcave} and \ref{thm:HH-3-dim-case}, it is natural to conjecture 
that the maximum in the right-hand side in Theorem \ref{thm:General_Result} is attained when 
$C$ is a cone, and $f$ is an affine function becoming $0$ at its base, thus yielding the following \emph{general result}:
\begin{conj}
Let $C\in\mathcal K^n$, $f:C\rightarrow[0,\infty)$ be concave, and $\phi:[0,\infty)\rightarrow[0,\infty)$ be convex with $\phi(0)=0$. Then
\[
\frac{1}{|C|}\int_C\phi(f(x))dx \leq n\int_0^1\phi\left(f(x_{C,f})\frac{2^{1/n}t}{2^{1/n}-1}\right)(1-t)^{n-1}dt. 
\]
Moreover, if $\phi$ is strictly convex, equality holds if and only if $C$ is a generalized cone, and $f$ is an affine function which becomes zero at the base of $C$.
\end{conj}

The paper is organized as follows. Section \ref{sec:Main_result_Proofs} is devoted to the proof of Theorem \ref{thm:General_Result}. Besides, we show also some key lemmas necessary to do it. For instance, in Lemma \ref{lem:ExistenceOfx_C} we show the existence of the point $x_{C,f}$, which is somewhat crucial for Theorem \ref{thm:General_Result}. Afterwards in Section \ref{sec:PlanarCase}, we compute the maximum established in Theorem \ref{thm:General_Result} when $C$ is a planar set, for particular cases of $\phi(t)$ (see Theorems \ref{thm:2_dim_case}, \ref{thm:Exp_2_dim_case}, \ref{thm:PhiConvNonLogConcave}). As a consequence of Theorem \ref{thm:2_dim_case}, we show Theorem \ref{thm:volume_sections}. Finally in Section \ref{sec:Higher_dim_case} we explore a little bit what happens in dimension $3$ or greater in Theorem \ref{thm:General_Result}.

\section{Proof of the main result}\label{sec:Main_result_Proofs}

Let us start this section by recalling that the \emph{Schwarz symmetrization} of $K\in\K^n$
with respect to $\lin(u)$, $u\in\R^n\setminus\{0\}$, is the set
\[
\sigma_u(K)=\bigcup_{t\in\R}\left(tu+r_t(B^n_2\cap u^\bot)\right),
\]
where $r_t\geq 0$ is such that $|K\cap(tu+u^\bot)|=r_t^{n-1}\kappa_{n-1}$, whenever 
$K\cap(tu+u^\bot)\neq\emptyset$, and $0$ otherwise.
It is well-known that $\sigma_u(K)\in\K^n$ and that $|\sigma_u(K)|=|K|$ (cf.~\cite[Section 9.3]{Gru} or \cite{Sch} for more details).
For every $K\in\mathcal K^n$ and $x\in\mathbb R^n\setminus\{0\}$, the \emph{support function} of $K$ at $x$
is defined by $h(K,x)=\sup\{x_1y_1+\cdots+x_ny_n:y\in K\}$.

Let $K,C\in\K^n$. The \emph{Minkowski addition} of $K$ and $C$ is defined as $K+C=\{x+y:x\in K,\,y\in C\}$.
The \emph{Brunn-Minkowski inequality} asserts that
\[
|(1-\lambda)K+\lambda C|^{\frac{1}{n}}\geq(1-\lambda)|K|^\frac{1}{n}+\lambda|C|^\frac{1}{n},
\]
for every $\lambda\in[0,1]$.
Moreover, equality holds if $K$ and $C$ are obtained one from each other by an
homothety or if $K$ and $C$ are contained in parallel hyperplanes
(see \cite{Ga} and the references therein for an insightful and complete study of this inequality).

Our first result shows that the if the Schwarz symmetrization of $C\in\K^n$ is a truncated cone, then $C$ is a generalized truncated cone. It will be needed for the equality cases of Theorem \ref{thm:General_Result}.
\begin{lemma}\label{lem:ConeGenCone}
Let $C\in\K^n$ be such that
\[
\sigma_{e_1}(C)=\conv((B_n\cap H)\cup(e_1+\rho(B_n\cap H))),
\]
where $H=\lin(\{e_2,\dots,e_n\})$ and $\rho\geq 0$. Then $C=\conv(\{C_0\cup(e_1+u+\rho C_0\})$, for some $u\in H$ and $C_0\in\K^n$ with $C_0\subset H$, i.e. $C$ is a generalized truncated cone.
\end{lemma}

\begin{proof}
Let $C'=\sigma_{e_1}(C)$, and let $M_t=C\cap\{(t,x_2,\dots,x_n)\in\R^n\}$ and $M'_t=C'\cap\{(t,x_2,\dots,x_n)\in\R^n\}$ for every $t\in[0,1]$.
Then $|M_t|=|M'_t|$ for every $t\in[0,1]$. On the one hand, it is very simple to check that 
$(1-\lambda)M_0'+\lambda M_1'=M'_\lambda$ for every $\lambda\in[0,1]$. On the other hand, the convexity of $C$ ensures that
$(1-\lambda)M_0+\lambda M_1\subset M_\lambda$. Using the Brunn-Minkowski inequality over subspaces parallel to $H$, then
\[
\begin{split}
    |M'_\lambda|^{\frac{1}{n-1}} & = |M_\lambda|^{\frac{1}{n-1}} \\
    & \geq (1-\lambda)|M_0|^\frac{1}{n-1}+\lambda|M_1|^\frac{1}{n-1} \\
    & =(1-\lambda)|M'_0|^\frac{1}{n-1}+\lambda|M'_1|^\frac{1}{n-1} = |M'_\lambda|^\frac{1}{n-1}.
\end{split}
\]
Therefore, we have equality in the Brunn-Minkowski inequality, and thus $M_\lambda$ are homothetic for every $\lambda\in[0,1]$.
Let $C_0:=M_0$ and, since $|M_1|=\rho^{n-1}|M_0|$, we have $M_1=e_1+u+\rho C_0$, for some $u\in H$.
Finally, observing again that 
\[
\begin{split}
(1-\lambda)M_0+\lambda M_1 & =(1-\lambda) C_0+\lambda(e_1+u+\rho C_0) \\
& \lambda e_1+\lambda u + ((1-\lambda)+\lambda \rho)C_0 \subset M_\lambda,
\end{split}
\]
and since $|\lambda e_1+\lambda u + ((1-\lambda)+\lambda \rho)C_0|=|M_\lambda|$, then
$M_\lambda=\lambda e_1+\lambda u + ((1-\lambda)+\lambda \rho)C_0$ for every $\lambda\in[0,1]$, concluding the proof.
\end{proof}

Next lemma is a crucial step for the proof of Theorem \ref{thm:General_Result}. We associate here, to every set $C'$ rotationally symmetric with respect to a line, a truncated cone $R$ with the same symmetry, with very special properties on the distribution of mass
of the set $R\setminus C'$.

\begin{lemma}\label{lem:TrunConeEqualVolumes}
Let $C'\in\K^n$ be rotationally symmetric w.r.t. $\lin(\{e_1\})$. Then there exists a truncated cone $R$ with bases orthogonal to $\lin(\{e_1\})$, rotationally symmetric w.r.t. $\lin(\{e_1\})$, and such that
\begin{enumerate}
    \item[i)] $|R|=|C'|$,
    \item[ii)] $h(R,-e_1)=h(C',-e_1)=t_0$, $h(R,e_1)=h(C',e_1)=t_1$ and
    \item[iii)] if $M_t'=\{(t,x_2,\dots,x_n)\in C'\}$, $M''_t=\{(t,x_2,\dots,x_n)\in R\}$, 
    then
    $M_t'\subset M_t''$ if and only if $t\in[t_0,t_0^*]\cup[t_1^*,t_1]$ and $M_t''\subset M_t'$ if and only if $t\in[t_0^*,t_1^*]$, 
    for some $t_0\leq t_0^*\leq t_1^*\leq t_1$ with
    \[
    \begin{split}
    & |(R\setminus C')\cap\{(t,x_2,\dots,x_n)\in\R^n:t\in[t_0,t_0^*]\}|\\
    & =|(R\setminus C')\cap\{(t,x_2,\dots,x_n)\in\R^n:t\in[t_1^*,t_1]\}|=\frac{|R\setminus C'|}{2}.
    \end{split}
    \]
\end{enumerate}
\end{lemma}

\begin{proof}
Let us define $v_{C'}(t)=\max\{v\geq 0:te_1+ve_2\in C'\}$, for every $t\in[t_0,t_1]$. Now, for every $m\in\R$, we define the truncated cone $R_m$ with bases orthogonal to $\lin(e_1)$ and whose centers belong to this line. We choose its $v_{R_m}(t)=r_m+m(t-t_0)$, $t\in[t_0,t_1]$, where $r_m\geq 0$ is chosen such that $|R_m|=|C|$. It is clear that the slope $m$ ranges in $m\in[-m_0,m_0]$, where $m_0$ is given by the solution to
\[
c=|C'|=\int_{t_0}^{t_1}\kappa_{n-1}(m_0(t-t_0))^{n-1}dt=\frac{\kappa_{n-1}(t_1-t_0)^n m_0^{n-1}}{n},
\]
i.e. $m_0=(cn/(\kappa_{n-1}(t_1-t_0)^n))^{1/(n-1)}$. Notice that on the one hand, since $C'$ is convex then $v_{C'}$ is concave,
and on the other hand $v_{R_m}$ is the graph of a line, for every $m$. 
Since $c=|R_m|$, it is clear that $v_{C'}$ and $v_{R_m}$ touch in at least one point for $t\in[t_0,t_1]$ (otherwise, $R_m$ or $C'$ would be contained in the other set, thus having that $|C'|<|R_m|$ or $|R_m|<|C'|$, 
contradicting the fact that $c=|R_m|$). Moreover, since
$R_{m_0}$ is a cone with vertex at $t=t_0$, by continuity arguments there exists $t^{m_0}_*\in[t_0,t_1]$ 
such that $v_{R_{m_0}}(t)\leq v_{C'}(t)$ for $t\in[t_0,t^{m_0}_*]$ and $v_{R_{m_0}}(t)\geq v_{C'}(t)$ for $t\in[t^{m_0}_*,t_1]$
(see Figure \ref{fig:Fig1}). 
\begin{figure}
    \centering
    \includegraphics[width=10cm]{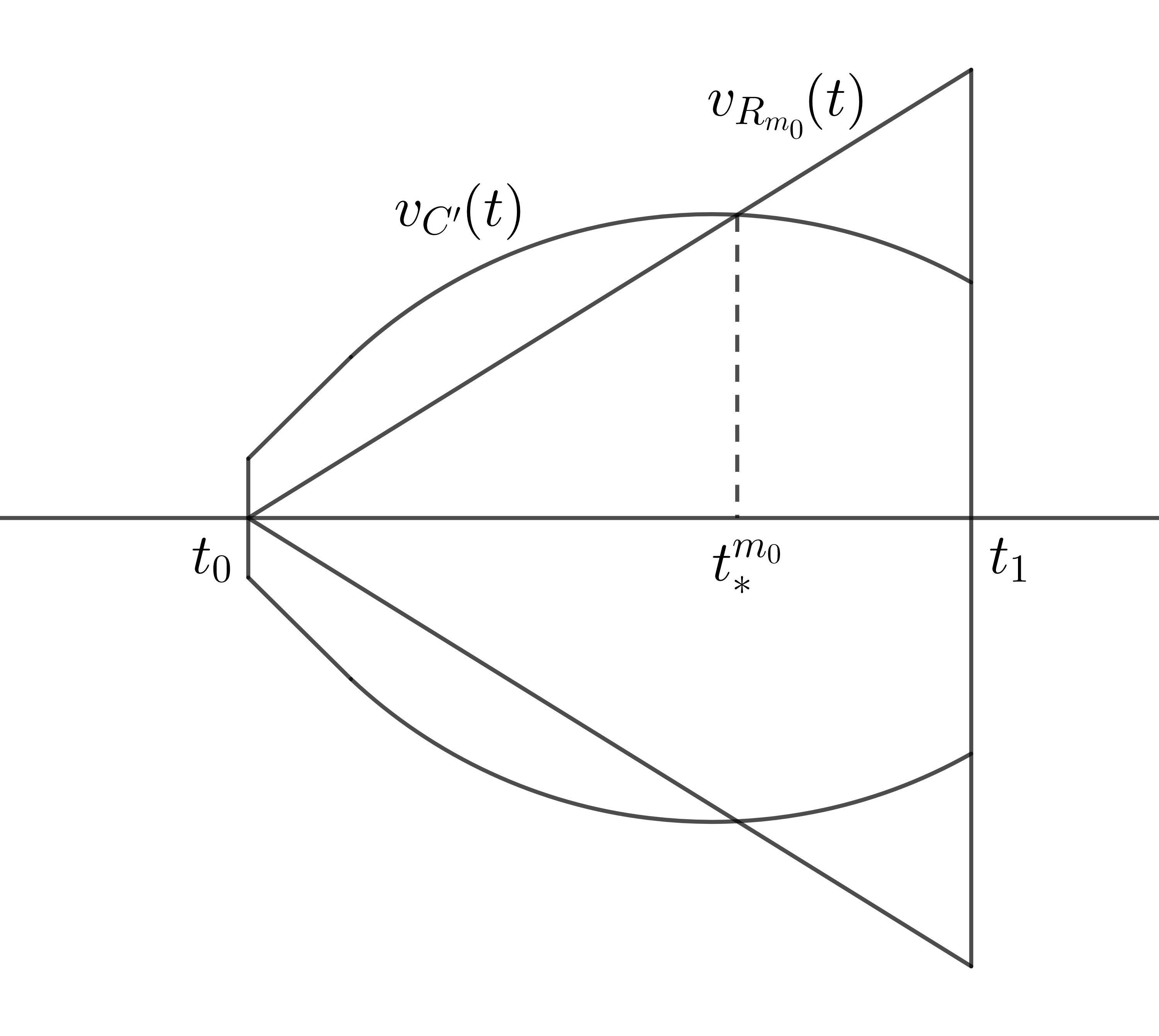}
    \caption{The function $v_{C'}$ is the graph of $C'$ whereas $v_{R_{m_0}}$ is the one of the cone $R_{m_0}$.}
    \label{fig:Fig1}
\end{figure}
Analogously, there exists $t^{-m_0}_*\in[t_0,t_1]$ such that $v_{R_{-m_0}}(t)\geq v_{C'}(t)$ for $t\in[t_0,t^{-m_0}_*]$ and  $v_{R_{-m_0}}(t)\leq v_{C'}(t)$ for $t\in[t^{-m_0}_*,t_1]$. 
Equivalently, when $m=m_0$, $M''_t\subset M'_t$ for $t\in[t_0,t^{m_0}_*]$, and $M'_t\subset M''_t$ otherwise, and when $m=-m_0$, $M'_t\subset M''_t$ for $t\in[t_0,t^{-m_0}_*]$, and $M''_t\subset M'_t$ for $t\in[t^{-m_0}_*,t_1]$.

Once $m$ starts increasing from $-m_0$ towards $m_0$, while $|R_m|=c$, by continuity
arguments there exist a maximal interval $[m_1,m_2]$ with $-m_0\leq m_1<m_2\leq m_0$, such that $v_{C'}$ and $v_{R_m}$ intersect in two points,
$t^m_0,t^m_1$ with $t_0\leq t^m_0\leq t^m_1\leq t_1$, for every $m\in[m_1,m_2]$ (see Figure \ref{fig:Fig2}).
\begin{figure}
    \centering
    \includegraphics[width=10cm]{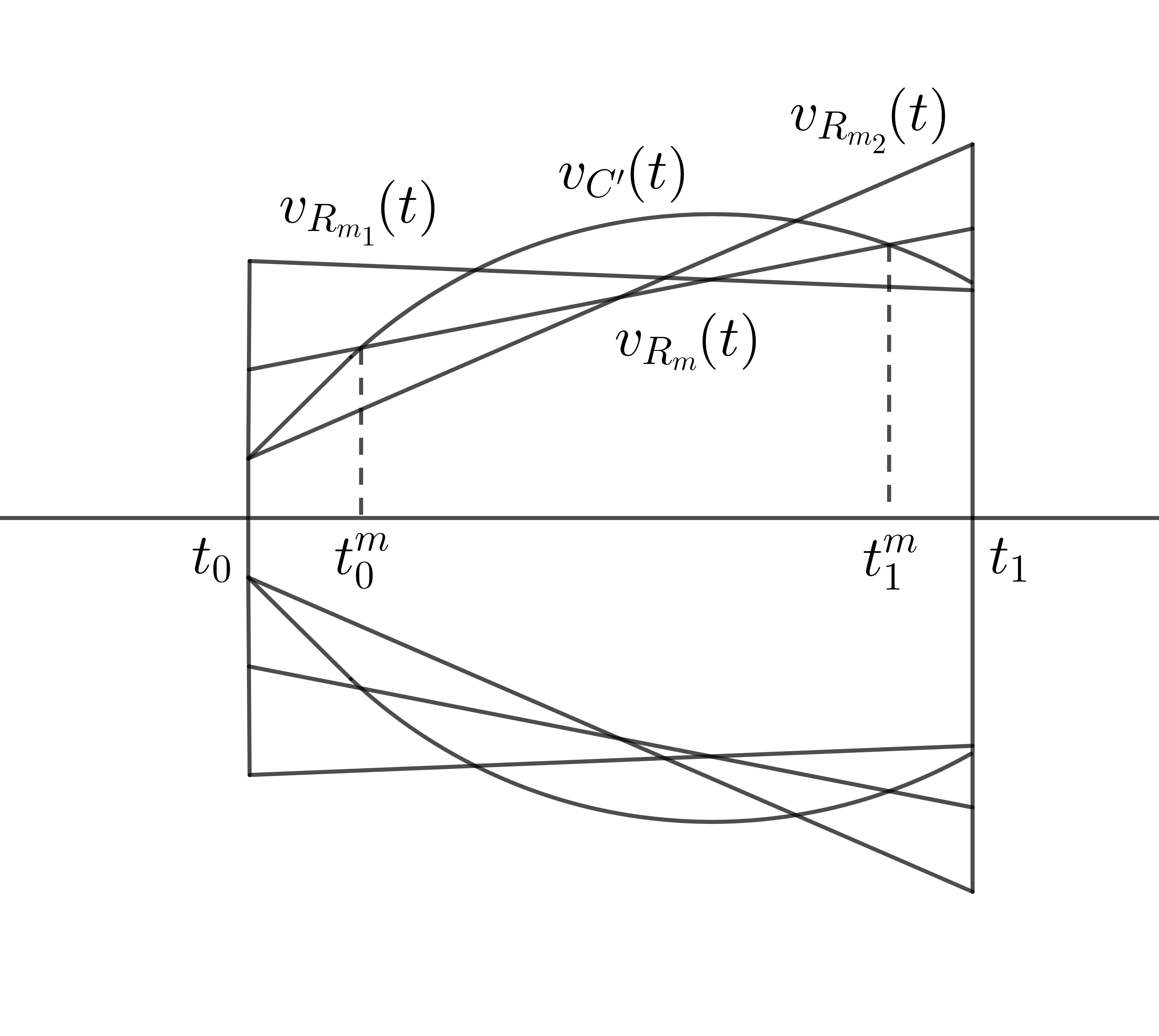}
    \caption{$v_{R_{m_1}}$ and $v_{R_{m_2}}$ are the truncated cones with largest and smallest slopes $m_1$ and $m_2$
    such that $v_{R_{m_i}}$ and $v_{C'}$ intersect in two points in $[t_0,t_1]$. Notice that $A_0(m)$ (resp. $A_1(m)$) coincides with the area between $v_{R_m}$ and $v_{C'}$ in $t\in[t_0,t^m_0]$ (resp. $t\in[t^m_1,t_1]$).}
    \label{fig:Fig2}
\end{figure}
Moreover, since both $R_m$ and $t_0^m$ change continuously on $m$, $A_0(m)=|(R_m\setminus C')\cap\{(t,e_2,\dots,e_n):t\in[t_0,t_0^m]\}|$ is continuous, ranging from $A_0(-m_0)=|R_{-m_0}\setminus C'|$ to $A_0(m_0)=0$ (see Figure \ref{fig:Fig2}). 
Analogously, $A_1(m)=|(R_m\setminus C')\cap\{(t,e_2,\dots,e_n):t\in[t_1^m,t_1]\}|$ is continuous too, ranging from $A_1(-m_0)=0$  to $A_1(m_0)=|R_{m_0}\setminus C'|$. 
Since $v_{C'}(t)\leq v_{R_m}(t)$ 
if and only if $t\in[t_0,t_0^m]\cup[t_1^m,t_1]$
then 
\begin{equation}\label{eq:volume_split}
\begin{split}
|(R_m\setminus C')\cap\{(t,e_2,\dots,e_n):t\in[t_0,t_0^m]\}\cup (R_m\setminus C')\cap\{(t,e_2,\dots,e_n):t\in[t_1^m,t_1]\}| & \\ =|R_m\setminus C'|&
\end{split}
\end{equation}
for every $m\in[m_1,m_2]$.
By the Bolzano theorem, there exists $m_{**}\in[m_1,m_2]$, such that 
\[
\begin{split}
& |(R_{m_{**}}\setminus C')\cap\{(t,e_2,\dots,e_n):t\in[t_0,t_0^{m_{**}}]\}| = A_0(m_{**})=A_1(m_{**}) \\
& =|(R_{m_{**}}\setminus C')\cap\{(t,e_2,\dots,e_n):t\in[t_1^{m_{**}},t_1]\}|,
\end{split}
\]
and by \eqref{eq:volume_split} they coincide are also equal to $|R_{m_{**}}\setminus C'|/2$.
Choosing $R=R_{m_{**}}$ concludes the lemma.
\end{proof}

The \emph{graph} of $f:\R^n\rightarrow[0,\infty)$ is defined by $G(f)=\{(x,f(x)):x\in\R^n\}$. 
The \emph{epigraph} of a concave function $f:\R^n\rightarrow[0,\infty)$, $\mathrm{epi}(f)=\{(x,\mu)\in\R^n\times\R:\mu\leq f(x)\}$ is a convex set in $\R^{n+1}$. If $C\in\K^n$, let $\partial C$ be the \emph{boundary} of $C$.

Below we provide a very simple observation, which was essentially proven in \cite{GoMe}, and establishes an interesting
property for affine functions bounding from above a concave function.

\begin{lemma}\label{lem:AffineGoverF}
Let $C\in\K^n$, $f:C\rightarrow[0,\infty)$ be concave. If $x_0\in \inter(C)$, then there exists an affine function $g$, with $g:C\rightarrow[0,\infty)$ with $g(x)\geq f(x)$ for $x\in C$ and $g(x_0)=f(x_0)$. Moreover, there exists $L\in\mathrm{Gr}(n-1,\mathbb R^n)$ (which after a rotation of $C$ we may assume to coincide
with $L=\lin(\{e_2,\dots,e_n\})$) such that
\[
(x_0,f(x_0))+L \subset G(g)\cap((x_0,f(x_0))+\lin(\{e_1,\dots,e_n\})).
\]
\end{lemma}

\begin{proof}
Since $(x_0,f(x_0))\in\partial(\mathrm{epi}(f))$, there exists an affine function $g$ such that $H=\mathrm{G}(g)$ supports $\mathrm{epi}(f)$ at $(x_0,f(x_0))$. In particular, $f(x)\leq g(x)$ for $x\in C$ and $f(x_0)=g(x_0)$. Since $H$ and $(x_0,f(x_0))+\lin(\{e_1,\dots,e_n\})$ are hyperplanes in $\R^{n+1}$, they intersect in a subspace of dimension at least $n-1$,
i.e. there exists a subspace $L\in\mathrm{Gr}(n-1,\mathbb R^n)$ such that 
\[
(x_0,f(x_0))+L\times\{0\}\subset H\cap((x_0,f(x_0))+\lin(\{e_1,\dots,e_n\})),
\]
as desired. Rotating $C$ appropriately, we can suppose that $L=\lin(\{e_2,\dots,e_n\})$, concluding the proof.
\end{proof}

The lemma below is also crucial step in the proof of Theorem \ref{thm:General_Result}. Given a set $C$ and a concave function $f$ with domain $C$, 
We show the existence of a point $x_{C,f}$ with very special properties with respect to the mass distribution of the truncated cone $R$ (given in Lemma \ref{lem:TrunConeEqualVolumes}) around $x_{C,f}$ associated to the Schwarz symmetrization of $C$ with respect to the line given in Lemma \ref{lem:AffineGoverF}. 

\begin{lemma}\label{lem:ExistenceOfx_C}
Let $C\in\K^n$ and $f:C\rightarrow[0,\infty)$ be concave. Then there exist
a point $x_{C,f}\in C$ such that (after rotating $C$ as in Lemma \ref{lem:AffineGoverF}) the point $P_{\lin(e_1)}(x_{C,f})=(t_{C,f},0,\dots,0)$ fulfills
\[
\int_{t_0}^{t_{C,f}}|M''_t|dt=\int^{t_1}_{t_{C,f}}|M''_t|dt=\frac{|R|}{2},
\]
where $R$ is the truncated cone given in Lemma \ref{lem:TrunConeEqualVolumes} with respect to $C'=\sigma_{e_1}(C)$ and $t_0$, $t_1$ and $M''_t$ are defined as in Lemma \ref{lem:TrunConeEqualVolumes}.
\end{lemma}


\begin{proof}
Let us fix a point $x_0\in C$. By Lemma \ref{lem:AffineGoverF}, there exists $L\in\mathrm{Gr}(n-1,\mathbb R^n)$ (after a rotation of $C$, assume that $L=\lin(\{e_2,\dots,e_n\})$) such that $(x_0,f(x_0))+L\times\{0\} \subset H\cap (x_0,f(x_0))+\lin(\{e_1,\dots,e_n\})$, where $H$ is a supporting hyperplane to $\mathrm{epi}(f)$ at $(x_0,f(x_0))$. Lemma \ref{lem:TrunConeEqualVolumes} gives us, for this choice of $L$, and for $C':=\sigma_{e_1}(C)$, a truncated cone $R$. Let $(t_R,0,\dots,0)\in R$ be the point such that
\begin{equation}\label{eq:HalvingVolumeR}
\int_{t_0}^{t_R}|M''_t|dt=\int_{t_R}^{t_1}|M''_t|dt=\frac{|R|}{2}.
\end{equation}
where we are adopting all the notation from Lemma \ref{lem:TrunConeEqualVolumes}.

Let $s_*\in[t_0,t_1]$ be such that $P_{e_1}(x_0+s_*e_1)=(t_R,0,\dots,0)$. 
Let $K_t(f)=\{x\in C:f(x)\geq t\}$, for $t\in[0,\|f\|_\infty]$, be the level sets of $f$. We have now two different cases, either 
\begin{enumerate}
    \item[i)] $(x_0+s_*e_1+L)\cap K_{\|f\|_\infty}(f)\neq\emptyset$ or
    \item[ii)] $(x_0+s_*e_1+L)\cap K_{\|f\|_\infty}(f)=\emptyset$.
\end{enumerate}

In case of i), let $v\in L$ be such that $x_0+s_*e_1+v\in K_{\|f\|_\infty}(f)$. 
Then the hyperplane $(x_0+s_*e_1+v,f(x_0+s_*e_1+v))+\lin(\{e_1,\dots,e_n\})$ supports $\mathrm{epi}(f)$ at $(x_0+s_*e_1+v,f(x_0+s_*e_1+v))$. Thus, according to Lemma \ref{lem:AffineGoverF}, we can choose $x_{C,f}:=x_0+s_*e_1+v$ with the subspace $L$, which together with \eqref{eq:HalvingVolumeR} concludes the lemma.

In case of ii), notice that $(K_{t}(f))_t$ is a continuously decreasing family on $t\in[0,\|f\|_\infty)$, such that $K_{0}(f)=C$. Hence, there exists $t_*\in[0,\|f\|_\infty)$ for which $x_0+s_*e_1+L$ supports $K_{t_*}(f)$, namely, at $x_0+s_*e_1+v$, for some $v\in L$. Hence, we can find a supporting hyperplane $H$ to $\mathrm{epi}(f)$ at $(x_0+s_*e_1+v,f(x_0+s_*e_1+v))$, such that $(x_0+s_*e_1+v,f(x_0+s_*e_1+v))+L\times\{0\}\subset H\cap (x_0+s_*e_1+v,f(x_0+s_*e_1+v))+\lin(\{e_1,\dots,e_n\})$. Thus, choosing 
$x_{C,f}:=x_0+s_*e_1+v$ with the subspace $L$, together with \eqref{eq:HalvingVolumeR} concludes the lemma.
\end{proof}

Notice that $x_{C,f}$ is in general not unique. Indeed, if $f$ is affine, then we find a whole hyperplane of directions for which any such point would fulfill Lemma \ref{lem:ExistenceOfx_C} as well.

\begin{proof}[Proof of Theorem \ref{thm:General_Result}]
We start applying Lemma \ref{lem:ExistenceOfx_C} (and thus also lemma \ref{lem:TrunConeEqualVolumes}) to $C$ and $f$, then apply Lemma \ref{lem:AffineGoverF} to $f$ and $x_{C,f}$, and consider the notation defined on them. Let $f_0=f(x_{C,f})=g(x_{C,f})$. 

Since $\phi$ is convex with $\phi(0)=0$, for any $x_2>x_1>0$ we have that
\[
0\leq\frac{\phi(x_1)-0}{x_1-0}\leq\frac{\phi(x_2)-\phi(x_1)}{x_2-x_1},
\]
i.e. $\phi$ is non-decreasing. Using that $f(x)\leq g(x)$ for every $x\in C$ (see Lemma \ref{lem:AffineGoverF}), then
\begin{equation}\label{eq:phi_Increasing}
\int_C\phi(f(x))dx\leq\int_C\phi(g(x))dx.
\end{equation}
Since 
\[
(x_{C,f},f_0)+L\times\{0\} \subset H\cap((x_{C,f},f_0)+\lin(\{e_1,\dots,e_n\})),
\]
where $H=\mathrm{aff}(G(g))$ (see Lemma \ref{lem:AffineGoverF}), then
\begin{equation}\label{eq:GisEasy}
g(t,e_2,\dots,e_n)=f_0+\delta(t-t_{C,f})
\end{equation}
for every $t\in[t_0,t_1]$
and for some $\delta$ such that $f_0+\delta(t_i-t_{C,f})\geq 0$, $i=0,1$, i.e.
\begin{equation}\label{eq:deltaInterval}
\frac{f_0}{t_{C,f}-t_1}\leq\delta\leq\frac{f_0}{t_{C,f}-t_0}.
\end{equation}
Let $M_t=\{(t,x_2,\dots,x_n)\in C\}$, for every $t\in[t_0,t_1]$. Then 
\eqref{eq:GisEasy} and Fubini's formula imply that
\[
\int_C\phi(g(x))dx = \int_{t_0}^{t_1}\phi(f_0+\delta(t-t_{C,f}))|M_t|dt.
\]
Let us furthermore observe that $g(t,x_2,\dots,x_n)$ has the value \eqref{eq:GisEasy} for every $(t,x_2,\dots,x_n)\in\R^n$, and thus it gets the same value both in $M_t$ and $M_t'=\{(t,x_2,\dots,x_n)\in C'\}$, where $C'=\sigma_{e_1}(C)$. Moreover, since $|M_t|=|M'_t|$ for every $t\in[t_0,t_1]$, we have that
\[
\int_{t_0}^{t_1}\phi(f_0+\delta(t-t_{C,f}))|M_t|dt=
\int_{t_0}^{t_1}\phi(f_0+\delta(t-t_{C,f}))|M'_t|dt.
\]
We now remember that since $\phi$ is convex, then $\phi(t)+\phi(t')\leq\phi(t-\delta)+\phi(t'+\gamma)$,
for every $t-\delta\leq t\leq t'\leq t'+\gamma$, $t,t',\delta,\gamma\in\mathbb R$.

Using the notation as in Lemma \ref{lem:TrunConeEqualVolumes}, let $M^*_t=M_t'\cap M_t''$ and $M^{**}_t=(M_t'\setminus M_t'')\cup(M_t''\setminus M_t')$. Since
\[
\begin{split}
&\int_{t_0}^{t_1}\phi(f_0+\delta(t-t_{C,f}))|M'_t|dt = \\
&\int_{t_0}^{t_1}\phi(f_0+\delta(t-t_{C,f}))|M^*_t|dt
+\int_{t^*_0}^{t^*_1}\phi(f_0+\delta(t-t_{C,f}))|M^{**}_t|dt,
\end{split}
\]
we just need to bound the last integral above. 
Notice that $\phi(f_0+\delta(t-t_{C,f}))\leq \phi(f_0+\delta(t_{\max}-t_{C,f}))$ for some $t_{\max}\in[t_0^*,t_1^*]$ and every $t\in[t_0^*,t_1^*]$. Analogously, we have that 
$\phi(f_0+\delta(t'-t_{C,f}))\geq 
\phi(f_0+\delta(t'_{\min}-t_{C,f}))$ for every 
$t'\in[t_0,t_0^*]$ and some $t'_{\min}\in[t_0,t_0^*]$, and
$\phi(f_0+\delta(t''-t_{C,f}))\geq \phi(f_0+\delta(t''_{\min}-t_{C,f}))$, for every $t''\in[t_1^*,t_1]$ and some $t''_{\min}\in[t_1^*,t_1]$.
Moreover, since $|C'|=|R|$, in particular $|C'\setminus R|=|R\setminus C'|$. 
Using that $t\rightarrow \phi(f_0+\delta(t-t_{C,f}))$ is convex too, then
\[
\begin{split}
& \int_{t_0^*}^{t_1^*}\phi(f_0+\delta(t-t_{C,f}))|M_t^{**}|dt \\
& \leq \phi(f_0+\delta(t_{\max}-t_{C,f}))\int_{t_0^*}^{t_1^*}|M_t^{**}|dt\\
& =2\phi(f_0+\delta(t_{\max}-t_{C,f}))\frac{|R\setminus C'|}{2}\\
& \leq (\phi(f_0+\delta(t'_{\min}-t_{C,f}))+\phi(f_0+\delta(t''_{\min}-t_{C,f}))) \frac{|R\setminus C'|}{2}\\
&=\phi(f_0+\delta(t'_{\min}-t_{C,f}))\int_{t_0}^{t_0^*}|M_t^{**}|dt +\phi(f_0+\delta(t''_{\min}-t_{C,f}))\int_{t_1^*}^{t_1}|M_t^{**}|dt\\
& \leq \int_{t_0}^{t_0^*}\phi(f_0+\delta(t'-t_{C,f}))|M_t^{**}|dt' +\int_{t_1^*}^{t_1}\phi(f_0+\delta(t''-t_{C,f}))|M_t^{**}|dt''.
\end{split}
\]
Hence, we have proven that
\[
\begin{split}
&\int_{t_0}^{t_1}\phi(f_0+\delta(t-t_{C,f}))|M'_t|dt \leq \int_{t_0}^{t_1}\phi(f_0+\delta(t-t_{C,f}))|M^*_t|dt \\
& +\int_{t_0}^{t_0^*}\phi(f_0+\delta(t'-t_{C,f}))|M_t^{**}|dt' +\int_{t_1^*}^{t_1}\phi(f_0+\delta(t''-t_{C,f}))|M_t^{**}|dt''\\
& =\int_{t_0}^{t_1}\phi(f_0+\delta(t-t_{C,f}))|M''_t|dt. 
\end{split}
\]
In the next (and last) step, there is a dichotomy. Either $\delta\geq 0$ or $\delta\leq 0$. Hence, assume $\delta\geq 0$ (the other case can be proven analogously). Let $\delta_{MAX}=f_0/(t_{C,f}-t_0)$, thus having $\delta\in[0,\delta_{MAX}]$.
Using Lemma \ref{lem:ExistenceOfx_C}, and since $\delta\rightarrow\phi(f_0+\delta(t-t_{C,f}))$ is a convex function too, we have that 
\[
\phi(f_0+\delta(t-t_{C,f}))+\phi(f_0+\delta(t'-t_{C,f})) \leq 
\phi(f_0+\delta_{MAX}(t-t_{C,f}))+\phi(f_0+\delta_{MAX}(t'-t_{C,f}))
\]
for every $t\in[t_0,t_{C,f}]$ and every $t'\in[t_{C,f},t_1]$. 
Thus
\[
\phi(f_0+\delta(t-t_{C,f})) -\phi(f_0+\delta_{MAX}(t-t_{C,f}))\leq -\phi(f_0+\delta(t'-t_{C,f}))
+\phi(f_0+\delta_{MAX}(t'-t_{C,f}))
\]
for every $t\in[t_0,t_{C,f}]$ and every $t'\in[t_{C,f},t_1]$. By continuity and compactness standard arguments, there exist $t_*\in[t_0,t_{C,f}]$ and $t_*'\in[t_{C,f},t_1]$ such that
\[
\phi(f_0+\delta(t-t_{C,f})) -\phi(f_0+\delta_{MAX}(t-t_{C,f}))\leq
\phi(f_0+\delta(t_*-t_{C,f})) -\phi(f_0+\delta_{MAX}(t_*-t_{C,f}))
\]
for every $t\in[t_0,t_{C,f}]$
as well as
\[
-\phi(f_0+\delta(t_*'-t_{C,f}))
+\phi(f_0+\delta_{MAX}(t_*'-t_{C,f}))\leq
-\phi(f_0+\delta(t'-t_{C,f}))
+\phi(f_0+\delta_{MAX}(t'-t_{C,f}))
\]
for every $t'\in[t_{C,f},t_1]$. Therefore
\[
\begin{split}
& \int_{t_0}^{t_{C,f}}(\phi(f_0+\delta(t-t_{C,f}))-
    \phi(f_0+\delta_{MAX}(t-t_{C,f})))|M_t''|dt\\
    & \leq\int_{t_0}^{t_{C,f}}(\phi(f_0+\delta(t_*-t_{C,f}))-
    \phi(f_0+\delta_{MAX}(t_*-t_{C,f})))|M_t''|dt \\
    & =(\phi(f_0+\delta(t_*-t_{C,f}))-
    \phi(f_0+\delta_{MAX}(t_*-t_{C,f})))\int_{t_0}^{t_{C,f}}|M_t''|dt\\
    & =(\phi(f_0+\delta(t_*-t_{C,f}))-
    \phi(f_0+\delta_{MAX}(t_*-t_{C,f})))\frac{|R|}{2}\\
    & \leq (-\phi(f_0+\delta(t_*'-t_{C,f}))
+\phi(f_0+\delta_{MAX}(t_*'-t_{C,f})))\frac{|R|}{2}\\
    & =(-\phi(f_0+\delta(t_*'-t_{C,f}))
+\phi(f_0+\delta_{MAX}(t_*'-t_{C,f})))\int_{t_{C,f}}^{t_1}|M_{t'}''|dt'\\
    & \leq \int_{t_{C,f}}^{t_1}(-\phi(f_0+\delta(t'-t_{C,f}))
+\phi(f_0+\delta_{MAX}(t'-t_{C,f})))|M_{t'}''|dt',
\end{split}
\]
i.e. 
\[
\begin{split}
& \int_{t_0}^{t_1}\phi(f_0+\delta(t-t_{C,f}))|M''_t|dt\\
& =\int_{t_0}^{t_{C,f}}\phi(f_0+\delta(t-t_{C,f}))|M''_t|dt+ \int_{t_{C,f}}^{t_1}\phi(f_0+\delta(t'-t_{C,f}))|M''_{t'}|dt'\\
& \leq 
\int_{t_0}^{t_{C,f}}\phi(f_0+\delta_{MAX}(t-t_{C,f}))|M''_t|dt+ \int_{t_{C,f}}^{t_1}\phi(f_0+\delta_{MAX}(t'-t_{C,f}))|M''_{t'}|dt'\\
&=\int_{t_0}^{t_1}\phi(f_0+\delta_{MAX}(t-t_{C,f}))|M''_t|dt,
\end{split}
\]
for every $\delta\in[0,\delta_{MAX}]$. As mentioned above, for every $\delta\in[\delta_{MIN},0]$, one can prove analogously that
\[
 \int_{t_0}^{t_1}\phi(f_0+\delta(t-t_{C,f}))|M''_t|dt
 \leq \int_{t_0}^{t_1}\phi(f_0+\delta_{MIN}(t-t_{C,f}))|M''_t|dt.
\]
In other words,
\[
\begin{split}
& \int_{t_0}^{t_1}\phi(f_0+\delta(t-t_{C,f}))|M''_t|dt\\
& \leq \max_{\delta_*\in\{\delta_{MIN},\delta_{MAX}\}}\int_{t_0}^{t_1}\phi(f_0+\delta_*(t-t_{C,f}))|M''_t|dt.
\end{split}
\]
Therefore, we have shown that
\[
\int_C\phi(f(x))dx\leq \max_{1,2}\int_R\phi(g_i(x))dx,
\]
where $g_1$ and $g_2$ are affine functions becoming zero at one of the bases of $R$ (corresponding with the slopes $\delta_{MAX}$ and $\delta_{MIN}$, respectively) and such that $g_i(x_{C,f})=f(x_{C,f})$, $i=1,2$. Therefore, the general upper bound for the term $\int_C\phi(f(x))dx$ is 
$\max_{R}\int_R\phi(g_1(x))dx$, where $R$ is a truncated cone with $|R|=|C|$ and where $g_1$ is a affine function with $g_1(x_{C,f})=f(x_{C,f})$ which becomes zero in a base of $R$ (see that now it is unnecessary to mention both $g_1$ and $g_2$, since $R$ covers positive and negative slopes).
This concludes the proof.

For the equality case, since $\phi$ is strictly convex, then $\phi$ is also strictly increasing.
Indeed, if $\phi$ is strictly convex with $\phi(0)=0$, for any $x_2>x_1>0$ we have that
\[
0<\frac{\phi(x_1)-0}{x_1-0}<\frac{\phi(x_2)-\phi(x_1)}{x_2-x_1},
\]
i.e. $\phi$ is strictly increasing. 
This, together with \eqref{eq:phi_Increasing}, shows that $f$ has to be an affine function, more particularly, $f=g$.
Moreover, $\sigma_{e_1}(C)$ must be one of the truncated cones attaining the maximum above. Moreover, $f$ has to become zero in one of the bases of the truncated cone. Hence, by Lemma \ref{lem:ConeGenCone},
$C$ must be a generalized truncated cone, attaining also the maximum, such that $f$ becomes zero in one of the bases of $C$.


\end{proof}


Let us finish this section by noting how to derive \eqref{eq:reduced_formula} from the right-hand side of the inequality in Theorem \ref{thm:General_Result},
as well as all the values of the parameters.
\begin{rmk}\label{rmk:concrete_computation}
After a suitable change of variables, we can assume that $R$ is rotationally symmetric with respect to $\lin(e_1)$, that
$h(R,-e_1)=0$ and $h(R,e_1)=1$. Then we can express $R=\{(t,x_2,\dots,x_n):t\in[0,1],(x_2,\dots,x_n)\in v_R(t)B^{n-1}_2\}$,
where $v_R(t)=\max\{s\geq 0:te_1+se_2\in R\}$. Since $R$ is a truncated cone, then we can parametrize 
$v_R(t)=r_m+mt$, for some real numbers $r_m$ and $m$ that can be computed. On the one hand, since $c=|C|=|R|$,
then the maximum slope $m_0$ such that $m\in[-m_0,m_0]$ is given by the extreme case in which $R$ is a cone, and thus when
\[
c=|R|=\int_0^1\kappa_{n-1}(m_0t)^{n-1}dt=\kappa_{n-1}m_0^{n-1}\frac1n.
\]
On the other hand, if $m\in[-m_0,m_0]$, $r_m$ is computed such that
\[
c=|R|=\int_0^1\kappa_{n-1}(r_m+mt)^{n-1}dt=\kappa_{n-1}\left((r_m+m)^n-r_m^n)\right)\frac1{nm}.
\]
Moreover, assuming that $x_{R,g}=(t_m,0,\dots,0)$ and that $g$ is an affine function becoming $0$ at $t=0$ with $g(x_{R,g})=f(x_{C,f})=f_0$, where $t_m\in[0,1]$ is such that
\[
\int_0^{t_m}\kappa_{n-1}(r_m+mt)^{n-1}dt=\int_{t_m}^1\kappa_{n-1}(r_m+mt)^{n-1}dt
\]
i.e. $(r_m+mt_m)^n-r_m^n=(r_m+m)^n-(r_m+mt_m)^n$.
Finally, we conclude that
\[
\int_R\phi(g(x))dx=\int_0^1\phi\left(\frac{f_0t}{t_m}\right)\kappa_{n-1}(r_m+mt)^{n-1}dt.
\]
\end{rmk}

\section{Planar case}\label{sec:PlanarCase}

Particularizing Theorem \ref{thm:General_Result} to the planar case and using Remark \ref{rmk:concrete_computation}, 
we easily get that $c=m_0$, $r_m=(c-m)/2$ and $t_m=\frac{-(c-m)+\sqrt{c^2+m^2}}{2m}$.
Thus we can write the right-hand side of its inequality as
\begin{equation}\label{eq:Part2dimCase}
\max_{m\in[-m_0,m_0]}\int_0^1\phi\left(\frac{f_0t}{t_m}\right)2(r_m+mt)dt.
\end{equation}
where $f_0=f(x_{C,f})$.

\begin{proof}[Proof of Theorem \ref{thm:2_dim_case}]
Let $c=|C|$ and $f_0=f(x_{C,f})$. Using the observation in \eqref{eq:Part2dimCase}, then Theorem \ref{thm:General_Result} applied to $C$, $f$ and $\phi(t)=t^{\alpha}$ says that
\[
\begin{split}
    \int_Cf(x)^\alpha dx & \leq 2\max_{m\in[-c,c]}\int_0^1\left(f_0\frac{t}{t_m}\right)^\alpha\left(\frac{c-m}{2}+mt\right)dt \\
    & =2f_0^\alpha\max_{m\in[-c,c]}\left(\frac{c-m}{2(\alpha+1)}+\frac{m}{\alpha+2}\right)\frac{1}{t_m^\alpha}.
\end{split}
\]

Notice that $t_m=(-(c-m)+\sqrt{c^2+m^2})/(2m)$, thus the maximum above rewrites as
\[
\begin{split}
cf_0^\alpha\max_{m\in[-c,c]} & \frac{1}{t_m^{\alpha+1}(t_m-1)}\left(\frac{2t_m^2-1}{2(\alpha+1)}+\frac{1-2t_m}{\alpha+2}\right)\\
& =cf_0^\alpha\max_{m\in[-c,c]}\varphi(t_m).
\end{split}
\]
Since 
\[
\varphi'(t_m)=-\frac{\alpha(2(t_m-2)t_m+1)t_m(\alpha(t_m-1)+2t_m-1)}{2(\alpha+1)(\alpha+2)(t_m-1)^2},
\]
thus $\varphi'(t_m)=0$ if and only if $t_m=1\pm 1/\sqrt{2}$ or $t_m=(a+1)/(a+2)$. Since $1/2<(\alpha+1)/(\alpha+2)$ and $\varphi'(1/2)=-2^{\alpha+1}\alpha^2/(\alpha^2+3\alpha+2)<0$ for every $\alpha\geq 1$, thus $\varphi(t_m)$ attains a local minimum at $t_m=(\alpha+1)/(\alpha+2)$. Thus the maximum of $\varphi(t_m)$ is attained either at $t_m=1-1/\sqrt{2}$ or $t_m=1/\sqrt{2}$, i.e.
\[
cf_0^\alpha\max_{m\in[-c,c]}\varphi(t_m) = cf_0^\alpha\frac{2^{\frac{\alpha}{2}+1}}{\alpha+2}\max\{\frac{1}{(\sqrt{2}-1)^\alpha(\alpha+1)},1\}.
\]
Since $(\sqrt{2}-1)^\alpha(\alpha+1)$ is strictly decreasing in $\alpha\geq 1$, then the maximum above becomes
$\max\{1/(2(\sqrt{2}-1)),1\}=1/(2(\sqrt{2}-1))$, and hence it is always attained at $t_m=1-1/\sqrt{2}$, i.e. at $m=-c$.
Therefore
\[
\int_Cf(x)^\alpha dx \leq 2cf_0^\alpha\frac{\sqrt{2}^\alpha}{(\sqrt{2}-1)^\alpha(\alpha+1)(\alpha+2)},
\]
concluding the result.

In the case of equality, notice that the maximum computed above is attained if and only if $m=-c$. Thus, equality holds if and only if $C$ is a (2-dimensional) cone, i.e. a triangle, and in view of the equality cases of Theorem \ref{thm:General_Result}, $f$ is an affine function. Moreover, $f$ has value zero at one of the edges of $C$ since $m=-c$ and we chose above $f(t)=f_0t/t_{-c}$.
\end{proof}

A direct application of Theorem \ref{thm:2_dim_case} is Theorem \ref{thm:volume_sections}, which now we are able to show.

\begin{proof}[Proof of Theorem \ref{thm:volume_sections}]
By Fubini's formula, we have that
\[
|K|=\int_{P_HK}|K\cap(x+H^\bot)|dx.
\]
By the Brunn's Concavity Principle (see \cite[Prop.~1.2.1]{Giann}, see also \cite{Ga}) then
\[
f:H\rightarrow[0,\infty)\quad\text{where}\quad f(x):=|K\cap(x+H^\bot)|^{\frac{1}{n-2}}
\]
is a concave function.
After a suitable rigid motion, we assume that $H=\R^2\times\{0\}^{n-2}$. By Theorem \ref{thm:2_dim_case} then 
\[
\begin{split}
\int_{P_HK}f(x)^{n-2}dx & \leq\frac{2}{(n-1)n}\left(\frac{\sqrt{2}}{\sqrt{2}-1}\right)^{n-2}|P_HK|f(x_{P_HK,f})^{n-2}\\
& =\frac{2}{(n-1)n}\left(\frac{\sqrt{2}}{\sqrt{2}-1}\right)^{n-2}|P_HK||K\cap (x_{P_HK,f}+H^\bot)|,
\end{split}
\]
concluding the result.

We omit the technical details to the characterization of the equality case. However, it should be done exactly as the equality case in the proof of Theorem 1.1 in \cite{GoMe}.
\end{proof}


\begin{proof}[Proof of Theorem \ref{thm:Exp_2_dim_case}]
Let $c=|C|$ and $f_0=f(x_{C,f})$. Using the observation \eqref{eq:Part2dimCase} 
and applying Theorem \ref{thm:General_Result} to $C$, $f/f_0$ and $\phi(t)=e^t-1$, then
\[
\begin{split}
    \int_Ce^{\frac{f(x)}{f_0}}dx & -|C| = \int_C(e^{\frac{f(x)}{f_0}}-1)dx \\
    & \leq \max_{m\in[-c,c]} 2 \int_0^1(e^{\frac{t}{t_m}}-1)\left(\frac{c-m}{2}+mt\right)dt\\
    & =\max_{m\in[-c,c]}2t_m\left(e^{\frac{1}{t_m}}\left(\frac{c-m}{2}+mt_m\left(\frac{1}{t_m}-1\right)\right)-\frac{c-m}{2}+mt_m\right)-|C|.
\end{split}
\]
Since $t_m=(-c+m+\sqrt{c^2+m^2})/(2m)$ then $m=c(1-2t_m)/(2t_m(t_m-1))$, and thus
\[
\begin{split}
\max_{m\in[-c,c]}2t_m\left(e^{\frac{1}{t_m}}\left(\frac{c-m}{2}+mt_m\left(\frac{1}{t_m}-1\right)\right)-\frac{c-m}{2}+mt_m\right) & \\
 =\max_{m\in[-c,c]}\frac{c}{t_m-1} \left(e^{\frac{1}{t_m}}\left(\frac{2t_m^2-1}{2}+(1-2t_m)\left(1-t_m\right)\right)\right. & \\
 \left.-\frac{2t_m^2-1}{2}+(1-2t_m)t_m\right) & \\
 =\max_{m\in[-c,c]}\frac{c}{t_m-1} \left(e^{\frac{1}{t_m}}\left(\frac{2t_m^2-1}{2}+(1-2t_m)\left(1-t_m\right)\right)\right. & \\
 \left.-\frac{2t_m^2-1}{2}+(1-2t_m)t_m\right) =\max_{m\in[-c,c]}\varphi(t_m). &
\end{split}
\]
Since $m\in[-c,c]$, then $t_m\in[1-1/\sqrt{2},1/\sqrt{2}]$.
It is a tedious computation to check (possibly with use of a software like \emph{Mathematica}) that $\varphi(t_m)$ is strictly decreasing. Indeed 
\[
\varphi'(t_m)=c\frac{(2(t_m-2)t_m+1)(e^\frac{1}{t_m}(3(t_m-1)t_m+1)-3t_m^2)}{2(t_m-1)^2t_m^2}.
\]
The local extrema of $\varphi(t_m)$, given by $\varphi'(t_m)=0$, are $t_m=1-1/\sqrt{2}$ and the numerical approximations
$t_m\approx 0$.$73487$, $1$.$707106$. Since $\varphi'(1/2)=3-e^2<0$ and $1/\sqrt{2}<0$.$73487$,
we can conclude that $\varphi(t_m)$ is decreasing in $[1-1/\sqrt{2},1/\sqrt{2}]$.
Thus its maximum is attained when $t_m=1-1/\sqrt{2}$, i.e. $m=-c$, hence providing the result.

For the equality case, we are forced to have $m=-c$, i.e. $C$ is a cone (i.e. a triangle) and, due to the equality cases of Theorem \ref{thm:General_Result}, such that $f$ is an affine function. Moreover, $f$ has value zero at one of the edges of $C$ since $m=-c$ and we chose above $f(t)=f_0t/t_{-c}$. 
\end{proof}


\begin{proof}[Proof of Theorem \ref{thm:PhiConvNonLogConcave}]
Let $c=|C|$ and $f_0=f(x_{C,f})$. Using the observation \eqref{eq:Part2dimCase} 
and applying Theorem \ref{thm:General_Result} to $C$, $f/f_0$ and $\phi(t)=e^{t^2}-1$, then
\[
\begin{split}
\int_Ce^{\frac{f(x)^2}{f_0^2}}dx-|C|& =\int_C(e^{\frac{f(x)^2}{f_0^2}}-1)dx\\
& \leq \max_{m\in[-c,c]}\int_0^1(e^{\left(\frac{t}{t_m}\right)^2}-1)\left(\frac{c-m}{2}+mt\right)dt\\
& =\max_{m\in[-c,c]}\int_0^1e^{\left(\frac{t}{t_m}\right)^2}\left(\frac{c-m}{2}+mt\right)dt-|C|\\
& = 2\max_{m\in[-c,c]} \left(\frac{c-m}{2}\frac{\sqrt{\pi}}{2}\mathrm{erfi}\left(\frac{1}{t_m}\right)+(mt_m^2)\left(e^{\left(\frac{1}{t_m}\right)^2}-1\right)\right)-|C|.
\end{split}
\]
Since $t_m=(-(c-m)+\sqrt{c^2+m^2})/(2m)$, we get that $m=c(1-2t_m)/(2t_m(t_m-1))$, and thus
\[
\begin{split}
     2 & \max_{m\in[-c,c]} \left(\frac{c-m}{2}\frac{\sqrt{\pi}}{2}\mathrm{erfi}\left(\frac{1}{t_m}\right)+(mt_m^2)\left(e^{\left(\frac{1}{t_m}\right)^2}-1\right)\right) \\
     & =\max_{m\in[-c,c]}\frac{c}{2t_m(t_m-1)}\left(\frac{(2t_m^2-1)\sqrt{\pi}}{4}\mathrm{erfi}\left(\frac{1}{t_m}\right)+(1-2t_m)t_m^2\left(e^{(\frac{1}{t_m})^2}-1\right)\right) \\
     & =\max_{m\in[-c,c]}\varphi(t_m).
\end{split}
\]
Since $m\in[-c,c]$, then $t_m\in[1-1/\sqrt{2},1/\sqrt{2}]$, and thus 
\[
\begin{split}
\varphi'(t_m) & = \frac{c}{8(t_m-1)^2t_m^3} \left(\sqrt{\pi}(-2(t_m-1)t_m-1)t_m\mathrm{erfi}\left(\frac1{t_m}\right)+4(2(t_m-2)t_m+1)t_m^3 \right.\\
& \left.-2e^{\frac{1}{t_m^2}}(t_m(2t_m(2t_m((t_m-2)t_m-1)+5)-5)+1)\right).
\end{split}
\]
The local extrema of $\varphi(t_m)$, given by $\varphi'(t_m)=0$, is the numerical approximation $t_m\approx 0$.$722484$.
Since $\varphi'(1/2)=-2(\sqrt{\pi} \mathrm{erfi}(2) + 1 + e^4)<0$, we conclude that $\varphi(t_m)$ is strictly
decreasing in $t_m\in[1-1/\sqrt{2},1/\sqrt{2}]$, and thus its maximum is attained at $t_m=1-1/\sqrt{2}$, i.e. $m=-c$, which concludes the result.

We now suppose there is equality in the inequality above. First of all, we are forced to have $m=-c$, i.e, $C$ must be a $2$-dimensional cone (i.e. a triangle) and from the equality case of Theorem \ref{thm:General_Result}, $f$ is an affine function which has value zero at one of the edges of $C$, since we chose above $f(t)=f_0t/t_{-c}$.
\end{proof}


\section{Higher dimensional cases}\label{sec:Higher_dim_case}

Solving the remaining 1-parameter optimization problem in Theorem \ref{thm:General_Result} is, in general, a hard and challenging task.
Using Remark \ref{rmk:concrete_computation}, We clearly have $m_0=(cn/\kappa_{n-1})^{1/(n-1)}$. However, 
The problem arises when trying to compute $r_m$, since it is described as a solution to
\[
c = \frac{\kappa_{n-1}}{mn}\left(nmr_{m}^{n-1}+\cdots+nm^{n-1}r_m+m^n\right).
\]
This is a quite general polynomial of degree $n-1$ on the variable $r_m$.
Moreover, the value of $t_m$ is the only point in $[0,1]$ such that
\[
(r_m+mt_m)^n-r_m^n=(r_m+m)^n-(r_m+mt_m)^n,
\]
and thus
\[
t_m=\frac{1}{m}\left(\left(\frac{r_m^n+(r_m+m)^n}{2}\right)^{\frac{1}{n}}-r_m\right).
\]
Hence, using explicit values for $r_m$ and $t_m$ seems to be very technical for $n=4$ and intractable for $n\geq 5$.

In the $3$-dimensional case, we would have that $r_m$ is the root
\begin{equation}\label{eq:Value_r_m_dim3}
r_m=\frac{-m+\sqrt{\frac{4c}{\pi}-\frac{m^2}{3}}}{2}.
\end{equation}
Moreover, $t_m$ would be given by the formula
\begin{equation}\label{eq:Value_t_m_dim3}
t_m=\frac{1}{m}\left(\left(\frac{r_m^3+(r_m+m)^3}{2}\right)^{\frac{1}{3}}-r_m\right).
\end{equation}


\begin{proof}[Proof of Theorem \ref{thm:HH-3-dim-case}]
Applying Theorem \ref{thm:General_Result} to $C$, $f$ and $\phi(t)=t$ tells that
\[
\int_Cf(x)dx\leq \max_{m\in[-m_0,m_0]}\int_0^1\frac{f_0t}{t_m}\pi(r_m+mt)^2dt,
\]
where $f_0=f(x_{C,f})$, $c=|C|$, $m_0=\sqrt{3c/\pi}$, and $r_m$ and $t_m$ are given in \eqref{eq:Value_r_m_dim3} and \eqref{eq:Value_t_m_dim3}.
Therefore
\begin{equation}\label{eq:max_dim_3}
\begin{split}
    \int_Cf(x)dx & \leq \pi f_0 \max_{m\in[-m_0,m_0]}\frac{1}{t_m}\int_0^1 t(r_m+mt)^2dt \\
    & =\pi f_0 \max_{m\in[-m_0,m_0]} \frac{6r_m^2+8r_mm+3m^2}{t_m}.
\end{split}
\end{equation}
Setting $s=\sqrt{3c/\pi}m$, then $s\in[-1,1]$, and we have that
\[
r_m=\frac{1}{2}\sqrt{\frac c\pi}\left(-\sqrt{3}s+\sqrt{4-s^2}\right)=\frac{1}{2}\sqrt{\frac c\pi}r_{m(s)},
\]
and that
\[
t_m=\sqrt{\frac c\pi}\left(\left(\frac1{16}r_{m(s)}^3+\frac12\left(\frac12r_{m(s)}+\sqrt{3}s\right)^3\right)^\frac13-\frac12r_{m(s)}\right)
=\sqrt{\frac c\pi}t_{m(s)}
\]
Substituting in \eqref{eq:max_dim_3} we obtain
\[
\max_{s\in[-1,1]}\sqrt{3}\frac{c}{\pi}\frac{s\left(\frac32r_{m(s)}^2+4\sqrt{3}r_{m(s)}s+9s^2\right)}{t_{m(s)}}=\sqrt{3}\frac c\pi \max_{s\in[-1,1]} h(s).
\]
Since $h(s)$ is a decreasing function from $s=-1$ to $s\approx 0.52$, and increasing from $x\approx 0.52$ until $s=1$, and such 
that $h(-1)>h(1)$ (see Figure \ref{fig:dim_3}). Since the expression of $h'(s)$ is rather technical and lengthy, we leave out the details.
\begin{figure}
    \centering
    \includegraphics[width=6cm]{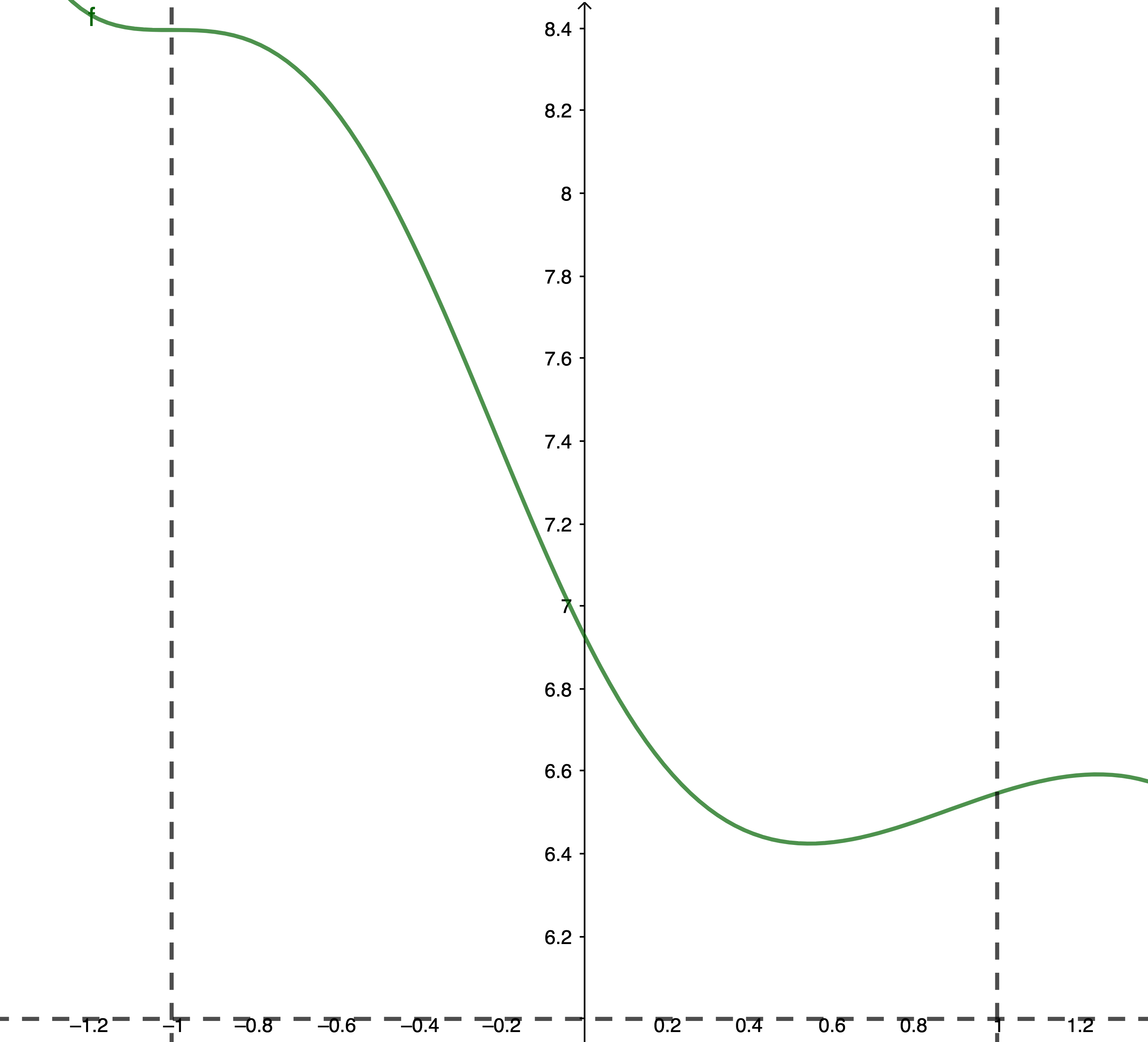}
    \caption{The function $h(s)$, $s\in[-1,1]$, which attains its maximum at $s=-1$.}
    \label{fig:dim_3}
\end{figure}
Therefore, the maximum in \eqref{eq:max_dim_3} is attained at $m=-\sqrt{3c/\pi}$, concluding the result.

For the equality case, notice that from the equality case of Theorem \ref{thm:General_Result} together with the fact shown above that we must have $m=-\sqrt{3c/\pi}$, then $C$ is a generalized cone. Moreover, the function $f$ must be an affine function which becomes zero at the base of $C$ (and not at the vertex of $C$), since we chose above $f(t)=f_0t/t_{-m_0}$. This concludes the proof.
\end{proof}

\emph{Acknowledgements:} I would like to thank the referee for the useful corrections, which helped in improving the readability of the paper.

\end{document}